\documentclass[a4paper, 11pt]{article}

\usepackage[english]{babel}
\usepackage[utf8]{inputenc} 
\usepackage{fullpage}
\usepackage{amsmath}
\usepackage{amssymb}
\usepackage{amsthm}
\usepackage{graphicx}
\usepackage{bbm}
\usepackage{url}
\usepackage{color}
\usepackage[colorlinks=true, allcolors=blue]{hyperref}

\usepackage{dsfont} 
\usepackage{enumitem}
\usepackage[normalem]{ulem} 

\numberwithin{equation}{section}

\theoremstyle{theorem}
\newtheorem{theorem}{Theorem}[section]
\newtheorem{proposition}[theorem]{Proposition}
\newtheorem{lemma}[theorem]{Lemma}

\theoremstyle{definition}
\newtheorem{definition}[theorem]{Definition}

\newtheorem{corollary}[theorem]{Corollary}

\newtheorem{assumption}[theorem]{Assumption}
\newtheorem{remark}[theorem]{Remark}

\newcommand{\ind}[1]{\mathds{1}_{#1}}

\newcommand{\stkout}[1]{\ifmmode\text{\sout{\ensuremath{#1}}}\else\sout{#1}\fi}

\title{Large deviations for fractional volatility models with non-Gaussian volatility driver
}

\author{Stefan Gerhold \and Christoph Gerstenecker \and Archil Gulisashvili}

\begin{document}



\maketitle

\begin{abstract}
	We study stochastic volatility models in which the volatility process is a function  of a continuous fractional stochastic process, 
		which is an integral transform of the solution of an SDE satisfying the Yamada-Watanabe condition. We establish a small-noise large deviation principle for the log-price, and, for a special case of our setup, obtain logarithmic call price
		asymptotics for large strikes.
\end{abstract}
\section{Introduction}\label{sec:introduction}


Recently, there has been a surge of interest in using stochastic Volterra equations
for financial modelling, with asymptotic approximations being a popular subject of research;
see the introductions of \cite{Gu18,Gu19} for many references.
While small-noise large deviations for such equations are well studied for Lipschitz coefficients~\cite{LiWaYaZh17,NuRo00,Zh08,Zh10}, results for processes that involve non-Lipschitz functions
in their dynamics are scarce.
 In the papers~\cite{FoGeSm19}
and~\cite{GeJaRoSh18}, concrete models with finite-dimensional parameter spaces are considered,
whereas~\cite{CePa20,FoZh17,Gu18,Gu19,Gu20_arxiv} study models where volatility is a function of a Gaussian
process. In the present paper, we assume that the volatility process is a function of
\begin{equation}\label{eq: V hat}
\hat{V}_t
	= \int_{0}^{t} K(t, s) U(V_s)\, ds,
\end{equation}
where~$U$ is a continuous non-negative function, assumptions on the kernel~$K$ will be specified below, and $V$ solves a one-dimensional SDE satisfying the Yamada-Watanabe condition. A (semi-)explicit generating function,
as is available  in the rough resp.\ fractional Heston models considered
in~\cite{FoGeSm19,GeJaRoSh18}, is not required.
Also, our process $\hat{V}$ 
is clearly non-Gaussian in general, which sets our results apart from the related papers
with Gaussian drivers mentioned above.
While our setup allows a lot of freedom in choosing the diffusion~$V$ and the other ingredients,
we note that truly rough models are not covered,
because~\eqref{eq: V hat} is a Lebesgue integral and not an integral
w.r.t.\ Brownian motion.
However, the models that we are considering may be rough at $t=0$ (see Remark 4.2).
The stock price is given by
\begin{align}\begin{split}
	dS_t
	& = S_t \sigma(\hat{V}_t) (\bar{\rho}\, dW_t + \rho\, dB_t),
	\quad 0 \leq t \leq T,\\
	 S_0
	& = 1.
\end{split}\end{align}
Here, $ B, W $ are independent standard Brownian motions,  $ \rho \in (-1, 1) $ and $ \bar{\rho} = \sqrt{1 - \rho^2} $. The extension to arbitrary $S_0>0$ is straightforward.
We now specify the conditions under which our main results, Theorems~\ref{thm:ldp for log-price (with W_T)}
and~\ref{thm:ldp for log-price} below, are valid.
Assumptions~\ref{ass:kernel definition}, \ref{ass:definition of u}
and~\ref{ass:assumptions for cir type diffusion}
 are in force throughout the paper. We note that the model
 defined in Section~2 of~\cite{BaDe20} is a special case of our model,
 but the aim of that paper is quite different from ours.

\begin{assumption}\label{ass:kernel definition}
	Let $ K $ be a kernel on $ [0, T]^2 $ satisfying the following conditions:
	\begin{enumerate}[label={\normalfont (\alph*)},ref={\normalfont \alph*}]
		\item
		\begin{align}\label{eq:kernel in L^2}
			\sup_{t \in [0, T]} \int_{0}^{T} K(t, s)^2\, ds
			< \infty.
		\end{align}
		
		\item The modulus of continuity of the kernel $ K $ in the space $ L^2[0, T] $ is defined as follows:
		\begin{align}\label{eq:moc of kernel}
			M(h)
			= \sup_{\{ t_1, t_2 \in [0, T]: |t_1 - t_2| \leq h \}} \int_{0}^{T} |K(t_1, s) - K(t_2, s)|^2\, ds,
			\quad 0 \leq h \leq T.
		\end{align}
		There exist constants $ c > 0 $ and $ r > 0 $ such that
		\begin{align}\label{eq:estimate for kernel's modulus of continuity}
			M(h)\leq c h^r
		\end{align}
		for all $ h \in [0, T] $.
		
		\item \label{it:condition (a) from [Archil 2018]}$ K(t, s) = 0 $ for all $ 0 \leq t < s \leq T $.
	\end{enumerate}
\end{assumption}

Then, $K$ is a Volterra kernel in the sense of \cite{Gu18} resp.\ \cite{Gu19}.
Of course, these conditions have been used earlier; e.g.,
(b) and~(c) are part of the definition of a Volterra type Gaussian
process in~\cite{Hu03,Hu03a}.
It is a standard fact that the associated integral operator
\begin{align}\label{eq:op K}
		\mathcal{K}(h)(t)
		= \int_{0}^{T} K(t, s) h(s)\, ds
	\end{align}
is  compact from $ L^2[0, T] $ into $ C[0, T] $; 
see e.g.\ Lemma~2 of \cite{Gu18} for a proof.
A standard example of a kernel satisfying Assumption~\ref{ass:kernel definition}
is the fractional kernel $\Gamma(H+\tfrac12)^{-1}(t-s)^{H-1/2},$
$0\leq s\leq t,$ with Hurst parameter $H\in(0,1).$ We note that~$\Gamma$ denotes the gamma function here, whereas later
we will use the letter~$\Gamma$ for the solution map of the
ODE~\eqref{eq:ode for v} below.

\begin{definition}\label{def:locally omega continuity}
	Let $ \omega $ be an increasing modulus of continuity on $ [0, \infty) $, that is $ \omega : \mathbb{R}_+ \to \mathbb{R}_+ $ is an increasing function such that $ \omega(0) = 0 $ and $ \lim\limits_{s \to 0} \omega(s) = 0 $. A function $ h $ defined on $ \mathbb{R} $ is called locally $ \omega $-continuous, if for every $ \delta > 0 $ there exists a number $ L(\delta) > 0 $ such that for all $ x, y \in [-\delta, \delta] $
	\begin{align}\label{eq:estimate w-continuous}
		|h(x) - h(y)|
		\leq L(\delta) \omega(|x - y|).
	\end{align}
\end{definition}

\begin{assumption}\label{ass:definition of u}
   The function $ U :\mathbb{R} \to [0, \infty) $ is continuous, and
    $ \sigma $ is a positive function on $ \mathbb{R}^+ $ that is locally $ \omega $-continuous for some modulus of continuity $ \omega $
	 as in Definition~\ref{def:locally omega continuity}.
\end{assumption}
The process $ V $ is assumed to solve the SDE
\begin{align}\begin{split}\label{eq:dV_t}
	dV_t
	&= \bar{b}(V_t)\, dt
	+ \bar{\sigma}(V_t)\, dB_t,
	\quad 0 \leq t \leq T,\\
	V_0
	&= v_0 > 0,
\end{split}\end{align}
where $ \bar{\sigma} $ and $ \bar{b} $ satisfy the Yamada-Watanabe
condition in Assumption~\ref{ass:assumptions for cir type diffusion}
below. A well-known example is the CIR process, where $\bar{\sigma}$ is the square
root function.


\begin{assumption}\label{ass:assumptions for cir type diffusion}
	~\begin{enumerate}[label={\normalfont (R\arabic*)},ref={\normalfont R\arabic*}]
		\item \label{it:(R1) Chiarini Fischer}
		The dispersion coefficient $ \bar{\sigma} : \mathbb{R} \to [0, \infty) $ is locally Lipschitz continuous on $ \mathbb{R} \backslash \{ 0 \} $, has sub-linear growth at $ \infty $, and $ \bar{\sigma}(0) = 0 $, while $ \bar{\sigma}(x) > 0 $ for all $ x \neq 0 $. Moreover, there exists a continuous increasing function $ \gamma : (0, \infty) \to (0, \infty) $ such that
		\begin{align}\label{eq:yamada watanabe condition}
			\int_{0 +}^{\infty} \frac{du}{\gamma(u)^2}
			= \infty
		\end{align}
		and
		\begin{align*}
			|\bar{\sigma}(x) - \bar{\sigma}(y)|
			\leq \gamma(|x - y|)
			\quad \text{for all } x, y \in \mathbb{R},\, x \neq y.
		\end{align*}
		Here, the sub-linear growth at $ \infty $ is understood in the sense that for every $ x_0 $ there exists a $ \mu $ such that for all $ x > x_0 $ we have
		\begin{align*}
			|\bar{\sigma}(x)|^2
			\leq \mu (1 + |x|^2).
		\end{align*}
		
		\item \label{it:(R2) Chiarini Fischer}
		The drift coefficient $ \bar{b} : \mathbb{R} \to \mathbb{R} $ is locally Lipschitz continuous, has sub-linear growth at $ \infty $, and $ \bar{b}(0) > 0 $.
	\end{enumerate}
\end{assumption}


Next, introducing a small-noise parameter~$\varepsilon>0$, we define the scaled version $ V^{\varepsilon} $ of the process $ V $ by
\begin{align}\label{eq:scaled cir type process}\begin{split}
	dV^{\varepsilon}_t
	& = \bar{b}(V_t^{\varepsilon})\, dt
	+ \sqrt{\varepsilon} \bar{\sigma}(V_t^{\varepsilon})\, dB_t,\\
	V_0^{\varepsilon}
	& = v_0 > 0,
\end{split}\end{align}
and the scaled stock price by
\begin{align}
	dS^{\varepsilon}_t
	= \sqrt{\varepsilon} S^{\varepsilon}_t \sigma(\hat{V}^{\varepsilon}_t) (\bar{\rho}\, dW_t + \rho\, dB_t).
\end{align}
Here, we write $ \hat{V}^{\varepsilon} $ for the process
\begin{align}\label{eq:V hat epsilon}
	\hat{V}^{\varepsilon}_t
	= \int_{0}^{t} K(t, s) U(V_s^{\varepsilon})\, ds.
\end{align}
The scaled log-price process $ X^{\varepsilon} = \log S^{\varepsilon} $, which is the process of interest
for our large deviations analysis, is now given by
\begin{align}\begin{split}
	dX^{\varepsilon}_t
	& = \sqrt{\varepsilon} \sigma(\hat{V}^{\varepsilon}_t)(\bar{\rho}\, dW_t + \rho\, dB_t)
	- \frac{1}{2} \varepsilon \sigma(\hat{V}^{\varepsilon}_t)^2\, dt,
	\quad 0 \leq t \leq T,\\
	X_0^\varepsilon
	& = 0,
\end{split}\end{align}
and the integral representation is as follows:
\begin{align}\label{eq:scaled log-price}
	X_t^{\varepsilon}
	= - \frac{1}{2} \varepsilon \int_{0}^{t} \sigma(\hat{V}_s^{\varepsilon})^2\, ds
	+ \sqrt{\varepsilon} \int_{0}^{t} \sigma(\hat{V}_s^{\varepsilon})\, d(\bar{\rho} W_s + \rho B_s),
	\quad 0 \leq t \leq T.
\end{align}


\begin{definition}\label{def:definition of integral operators}
    In addition to~$\mathcal K$ from~\eqref{eq:op K}, we
	 define the integral operators
	\begin{align*}
		\hat{\cdot} &: C[0, T] \to C[0, T],\\
		\check{\cdot} &: H_0^1[0, T] \to C[0, T],
	\end{align*}
	by
	\begin{align}
		\label{eq:f hat (within def)}
		\hat{f}(t)
		& = \int_{0}^{t} K(t, s)U(f(s))\, ds,
		\quad t \in [0, T],\\
		\label{eq:f check (within def)}
		\check{g}(t)
		& = \int_{0}^{t} K(t, s) U(v(s))\, ds,
		\quad t \in [0, T],
	\end{align}
	where $ v $ is the solution of the ODE
	\begin{align}\label{eq:ode for v}
		\dot{v}
		= \bar{b}(v)
		+ \bar{\sigma}(v) \dot{g},
		\quad v(0) = v_0.
	\end{align}
\end{definition}


Clearly, we have  $ \check{g} = \hat{v} $, where $ v $ solves the ODE \eqref{eq:ode for v}.
Moreover,  $\hat{f} = \mathcal{K}(U \circ f)$ and
		$\check{g} = \mathcal{K}(U \circ \Gamma(g))$,
	where $ \Gamma$ maps $g$ to the solution of~\eqref{eq:ode for v}.	
	By Assumption~\ref{ass:kernel definition} the integral operators of Definition~\ref{def:definition of integral operators} are well-defined. In fact, for our kernel $ K $, we get that $ \mathcal{K} : L^2[0, T] \to C[0, T] $.  Note that for $ h \in H_0^1[0, T] $, we have $ h \in C[0, T] $. 
	Further, for $ f \in H_0^1[0, T] $ we have $ U \circ f \in L^2[0, T] $ and for $ g \in H_0^1[0, T] $ we have $ U \circ v \in L^2[0, T] $. This can be easily seen using the fact that $ U $ is continuous and the input functions are continuous on a bounded interval and hence bounded themselves.


We can now state our main results.

\begin{theorem}\label{thm:ldp for log-price (with W_T)}	
	The family $ X^{\varepsilon}_T $ satisfies the small-noise large deviation principle (LDP) with speed $ \varepsilon^{-1} $ and good rate function $ I_T $ given by
	\begin{align}\begin{split}
		I_T(x)
		& = \inf_{f \in H_0^1} \Big[ \frac{T}{2} \Big(\frac{x - \rho \langle \sigma(\mathcal{K}(U \circ \Gamma(f))), \dot{f} \rangle}{\bar{\rho} \sqrt{\langle \sigma(\mathcal{K}(U \circ \Gamma(f)))^2, 1\rangle}}\Big)^2 + \frac{1}{2} \langle \dot{f}, \dot{f} \rangle \Big]\\
		& = \inf_{f \in H_0^1} \Bigg[ \frac{T}{2} \bigg(\frac{x - \rho \int_{0}^{T} \sigma(\int_{0}^{t} K(t, s) U(\Gamma(f)(s))\, ds) \dot{f}(t)\, dt}{\bar{\rho} \sqrt{\int_{0}^{T} \sigma(\int_{0}^{t} K(t, s)U(\Gamma(f)(s))\, ds)^2\, dt}}\bigg)^2 + \frac{1}{2} \int_{0}^{T} \dot{f}(t)^2\, dt \Bigg]
	\end{split}\end{align}
	for all $ x \in \mathbb{R} $, wherever this expression is finite. The validity of the LDP means that for every Borel subset $ \mathcal{A} $ of $ \mathbb{R} $, the following estimate holds, where  $\mathcal{A}^{\circ}$ and $\bar{\mathcal{A}}$ denote the interior resp.\ the closure of~$\mathcal{A}$:
	\begin{align}
		- \inf_{x \in \mathcal{A}^{\circ}} I_T(x)
		\leq \liminf_{\varepsilon \searrow 0} \varepsilon \log P(X_T^\varepsilon \in \mathcal{A})
		\leq \limsup_{\varepsilon \searrow 0} \varepsilon \log P(X_T^\varepsilon \in \mathcal{A})
		\leq - \inf_{x \in \bar{\mathcal{A}}} I_T(x).
	\end{align}
\end{theorem}

\begin{theorem}\label{thm:ldp for log-price}
	The family of processes $  X^{\varepsilon} $  satisfies the sample path LDP with speed $ \varepsilon^{-1} $ and good rate function $ Q $ given by
	\begin{align}\begin{split}
		Q(g)
		& = \inf_{f \in H_0^1} \Big[ \frac{1}{2} \int_{0}^{T} \Big(\frac{\dot{g}(t) - \rho \sigma(\mathcal{K}(U \circ \Gamma(f))(t)) \dot{f}(t)}{\bar{\rho} \sigma(\mathcal{K}(U \circ \Gamma(f))(t))}\Big)^2\, dt + \frac{1}{2} \int_{0}^{T} |\dot{f}(t)|^2\, dt \Big]\\
		& = \inf_{f \in H_0^1} \Bigg[ \frac{1}{2} \int_{0}^{T} \bigg(\frac{\dot{g}(t) - \rho \sigma(\int_{0}^{t} K(t, s) U(\Gamma(f)(s))\, ds) \dot{f}(t) }{\bar{\rho} \sigma(\int_{0}^{t} K(t, s) U(\Gamma(f)(s))\, ds)}\bigg)^2\, dt + \frac{1}{2} \int_{0}^{T} |\dot{f}(t)|^2\, dt \Bigg]
	\end{split}\end{align}
	for all $ g \in H_0^1[0, T] $, and by $ Q(g) = \infty $, for all $ g \in C[0, T] \backslash H_0^1[0, T] $. The validity of the LDP means that for every Borel subset $ \mathcal{A} $ of $ C[0, T] $, the following estimate holds:
	\begin{align}
		- \inf_{g \in \mathcal{A}^{\circ}} Q(g)
		\leq \liminf_{\varepsilon \searrow 0} \varepsilon \log P(X^\varepsilon \in \mathcal{A})
		\leq \limsup_{\varepsilon \searrow 0} \varepsilon \log P(X^\varepsilon \in \mathcal{A})
		\leq - \inf_{g \in \bar{\mathcal{A}}} Q(g).
	\end{align}
\end{theorem}

The structure of this paper is as follows. In Section~\ref{sec:small-noise ldps}, we  recall small-noise large deviations for SDEs satisfying the Yamada-Watanabe condition.  In Section~\ref{sec:small-noise proofs}, we prove the main results, i.e.\ the small-noise LDP for the log-price. In Section~\ref{sec:strike} we specialize
our model to obtain a convenient scaling property, and obtain large-strike
asymptotics for call prices from our small-noise LDP.
As mentioned above, Assumptions~\ref{ass:kernel definition}, \ref{ass:definition of u}
and~\ref{ass:assumptions for cir type diffusion} are supposed to be satisfied throughout the rest of the paper.
\section{LDPs for the driving processes}
\label{sec:small-noise ldps}

\subsection{Sample path LDP for the diffusion}

We apply a result of~\cite{ChFi14}, which is based on a representation formula for
functionals of Brownian motion obtained
in~\cite{BoDu98}, to obtain an LDP for $(\sqrt{\varepsilon}B, V^{\varepsilon}) $.
While the Yamada-Watanabe condition from Assumption~\ref{ass:assumptions for cir type diffusion}
covers virtually all one-dimensional diffusions that have been suggested in financial modelling,
we note that Assumption~\ref{ass:assumptions for cir type diffusion} could still be weakened,
if desired, e.g.\ by inspecting the proof of Theorem~4.3 in~\cite{BoDu98}.

If assumptions (H1)--(H6) of \cite{ChFi14} hold, then the family of processes $ (\sqrt{\varepsilon} B, V^{\varepsilon}). $ which satisfy the two-dimensional SDE
\begin{align}
	\begin{pmatrix}
		\sqrt{\varepsilon} dB_t\\
		dV_t^{\varepsilon}
	\end{pmatrix}
	=
	\begin{pmatrix}
		0\\
		\bar{b}(V_t^{\varepsilon})
	\end{pmatrix}
	dt
	+ \sqrt{\varepsilon}
	\begin{pmatrix}
		1\\
		\bar{\sigma}(V_t^{\varepsilon})
	\end{pmatrix}
	dB_t,
\end{align}
admits an LDP due to Theorem~1 in \cite{ChFi14}.
For $ V^{\varepsilon} $, (H1)--(H6) have been checked in \cite[pp.~1143--1144]{ChFi14}. For $ (\sqrt{\varepsilon}B, V^{\varepsilon}) $, the proofs are similar. The assumptions (H1)--(H3) are clearly satisfied. Let us check condition (H4), namely unique solvability of the control equation~(7) in \cite{ChFi14}.
Here, it is
\begin{align}\label{eq:control equation for two-dimensional problem}
	\begin{pmatrix}
		\varphi_1(t)\\
		\varphi_2(t)
		\end{pmatrix}
	= \begin{pmatrix}
		0\\
		v_0
	\end{pmatrix}
	+ \int_{0}^{t} \begin{pmatrix}
		0\\
		\bar{b}(\varphi_1(s))
	\end{pmatrix}
	ds
	+ \int_{0}^{t} \begin{pmatrix}
		1\\
		\bar{\sigma}(\varphi_1(s))
	\end{pmatrix}
	f(s)\, ds,
\end{align}
where $ f \in L^2[0, T] $ is the control function.
We also have $ \varphi_1, \varphi_2 \in C[0, T] $. It follows that the unique solution of \eqref{eq:control equation for two-dimensional problem} is given by $ \Gamma_{v_0}(f) = \begin{pmatrix}
	\int_{0}^{\cdot} f(s)\, ds\\
	\varphi_2
\end{pmatrix} $, where the function $ \varphi_2 $ is the unique solution of the equation
\begin{align}\label{eq:control ode}
	\varphi_2(t)
	= v_0 + \int_{0}^{t} \bar{b}(\varphi_2(s))\, ds
	+ \int_{0}^{t} \bar{\sigma}(\varphi_2(s)) f(s)\, ds,
	\quad t \in [0, T],
\end{align}
that exists by \cite[Proposition~1]{ChFi14}. 
This establishes condition (H4) in our setting.
Note at this point, that the ODE~\eqref{eq:control ode} above is formulated for $ f \in L^2[0, T] $   to match the notation of \cite{ChFi14}. Alternatively it can also be written, with a $ g \in H_0^1 $,
 and $ \dot{g} $ instead of $ f $, see \eqref{eq:ode for v}.
Condition (H5) for the second component of $ \Gamma_{v_0} $ was checked in \cite[p.~1144]{ChFi14}. For the first component, (H5) is true by the following simple fact.
\begin{lemma}
	The map $ f \mapsto \int_{0}^{\cdot} f(s)\, ds $ is continuous from $ \mathcal{B}_r $ into $ C[0, T] $, where $ \mathcal{B}_r $ is the closed ball of radius $ r>0 $ in $ L^2[0, T] $ endowed with the weak topology.
\end{lemma}
\begin{proof}
	If $f_n\in\mathcal{B}_r$ converges weakly to~$f$, then the convergence is uniform
	on compact subsets of $ L^2[0, T] $. Since $\{ \mathds{1}_{[0,t]} : 0\leq t\leq T\}$ is compact,
	we have
	\begin{align}
		\sup_{t \in [0, T]} \left| \int_{0}^{t} f(u)\, du - \int_{0}^{t} f_n(u)\, du \right|
		\to 0,
		\quad n \to \infty.
	\end{align}
\end{proof}

The tightness assumption~(H6) can be established as in~\cite{ChFi14}.
The verification, which is based on the sub-linear growth of~$\bar b$
and~$\bar \sigma$ and the uniform moment estimate
in Lemma~A.2 of~\cite{ChFi14}, is found on pp.\ 1137--1138 of~\cite{ChFi14}.
See also Section~4.2 of~\cite{ChFi14}.
%
%
Now, Theorem~1 of~\cite{ChFi14} implies the following assertion.

\begin{theorem}\label{thm:ldp for (B,V)}
	The family of processes $  (\sqrt{\varepsilon} B, V^{\varepsilon}) $  satisfies an LDP in the space $ C[0, T]^2 $  with speed $ \varepsilon^{-1} $ and good rate function $ I : C[0, T]^2 \to [0, \infty] $ given by
	\begin{align}\label{eq:rate function for (B,V) implicit}
		I(\varphi_1, \varphi_2)
		= \inf_{\{ f \in L^2[0, T] :\, \Gamma_{v_0}(f) = \begin{pmatrix}
			\varphi_1\\
			\varphi_2
		\end{pmatrix} \}}
	\frac{1}{2} \int_{0}^{T} f(t)^2\, dt,
	\end{align}
	whenever $ \{ f \in L^2[0, T] : \Gamma_{v_0}(f) = \begin{pmatrix}
	\varphi_1\\
	\varphi_2
	\end{pmatrix} \} \neq \emptyset $, and $ I(\varphi_1, \varphi_2) = \infty $ otherwise. Here, $ \Gamma_{v_0}(f) $ maps $ f $ to the solution of \eqref{eq:control equation for two-dimensional problem}.
\end{theorem}

	Note that Theorem~1 of \cite{ChFi14} actually gives a Laplace principle. But since the rate function is a good rate function (which is shown in \cite{ChFi14}), we also get an LDP with the same rate function. See Theorems~1.2.1 and 1.2.3 of \cite{DuEl97}.
The condition
$ \Gamma_{v_0}(f) = \begin{pmatrix}
	\varphi_1\\
	\varphi_2
\end{pmatrix} $
implies that $ \int_{0}^{t} f(s)\, ds = \varphi_1(t) $, or $ f(t) = \dot{\varphi}_1(t) $. Therefore
\begin{align*}
	\dot{\varphi}_2(t)
	= \bar{b}(\varphi_2(t))
	+ \bar{\sigma}(\varphi_2(t))\dot{\varphi}_1(t),
\end{align*}
and hence
\begin{align}\label{eq:phi dot two}
	\dot{\varphi}_1(t)
	= \frac{\dot{\varphi}_2(t) - \bar{b}(\varphi_2(t))}{\bar{\sigma}(\varphi_2(t))}.
\end{align}
Therefore, the following statement holds:
\begin{corollary}\label{cor: rate function for phi_1 absolutely continuous}
	For every $ \varphi_2 $ that is absolutely continuous on $ [0, T] $ with $ \varphi_2(0) = v_0 $
	\begin{align}\label{eq:rate function for (B,V) explicit}
		I\Big(\int_{0}^{\cdot} \frac{\dot{\varphi}_2(t) - \bar{b}(\varphi_2(t))}{\bar{\sigma}(\varphi_2(t))}\, dt, \varphi_2\Big) 
		= \frac{1}{2} \int_{0}^{T} \Big(\frac{\dot{\varphi}_2(t) - \bar{b}(\varphi_2(t))}{\bar{\sigma}(\varphi_2(t))}\Big)^2\, dt,
	\end{align}
	if the integral is finite, and $ I(\varphi_1, \varphi_2) = \infty $ in all the remaining cases.
\end{corollary}
\subsection{Sample path LDP for  $  (\sqrt{\varepsilon} B, \hat{V}^{\varepsilon}) $}

In this subsection we lift the sample path LDP in Theorem~\ref{thm:ldp for (B,V)} to one for the
family of processes we get when applying the ``hat" operator 
defined in~\eqref{eq:V hat epsilon} to $ V^{\varepsilon}.$

\begin{lemma}\label{lem:f to f hat is continuous}
	The mapping $ f \mapsto \hat{f} $ is  continuous  from the space $ C[0, T] $ into itself.
\end{lemma}

\begin{proof}
	For $ f \in C[0, T] $ and all $ t_1, t_2 \in [0, T] $,
	\begin{align*}
		|\hat{f}(t_1) - \hat{f}(t_2)|
		\leq M(|t_1 - t_2|)^{\frac{1}{2}} \Big(\int_{0}^{T} U(f(s))^2\, ds\Big)^{\frac{1}{2}}
		\leq C_f |t_1 - t_2|^{\frac{r}{2}}.
	\end{align*}
	The number $ r $ in the exponent of the last term comes from an estimate for the modulus of continuity of the kernel given by \eqref{eq:estimate for kernel's modulus of continuity}.	Here we used the local boundedness of
	the continuous function $ U $,  and also~\eqref{eq:moc of kernel}.
	Now, it is clear that the function $ \hat{f} $ is continuous on $ [0, T] $. It remains to prove the continuity of the mapping $ f \mapsto \hat{f} $ on $ C[0, T] $. Suppose $ f_k \to f $ in $ C[0, T] $. Then we have
	\begin{align}\label{eq:fhat - fhat_k in C_0}
		\| \hat{f} - \hat{f}_k \|_{C[0, T]}
		\leq \Big(\int_{0}^{T} |U(f(s)) - U(f_k(s))|^2\, ds\Big)^{\frac{1}{2}} \sup_{t \in [0, T]} \Big(\int_{0}^{T} K(t, s)^2\, ds\Big)^{\frac{1}{2}}.
	\end{align}
	Moreover,
	\begin{align*}
		C_0=\max \big\{ \| f \|_{C[0, T]}, \sup_{k} \| f_k \|_{C[0, T]} \big\}
		< \infty.
	\end{align*}
	It follows from Assumption~\ref{ass:kernel definition} and~\eqref{eq:fhat - fhat_k in C_0}   that there exists a constant $ C_1 $ for which
	\begin{align}\label{eq:estimate hat functions with original ones}
		\| \hat{f} - \hat{f}_k \|_{C[0, T]}
		\leq C_1 \sup_{s \in [0, T]} \big|U(f(s)) - U(f_k(s))\big|,
	\end{align}
	and the previous expression converges to zero by the uniform continuity of~$U$ on $[-C_0,C_0]$. This completes the proof.
\end{proof}

The next assertion establishes the LDP
 for $  (\sqrt{\varepsilon} B, \hat{V}^{\varepsilon}) $.

\begin{theorem}\label{thm:ldp for epsB, Vhat}
	The family of processes $ (\sqrt{\varepsilon}B, \hat{V}^{\varepsilon}) $  satisfies an LDP in the space $ C[0, T]^2 $ with speed $ \varepsilon^{-1} $ and good rate function given by
	\begin{align}
		\tilde{I}\big(\psi_1, \mathcal{K}(U \circ \Gamma(\psi_1))\big)
		= \frac{1}{2} \int_{0}^{T} \dot{\psi}_1(t)^2\, dt,
	\end{align}
	if the expression in \eqref{eq:phi dot two} exists, and $ \tilde{I}(\psi_1, \psi_2) = \infty $ otherwise.
	Here, $ \Gamma $ is the solution map of the one-dimensional ODE \eqref{eq:ode for v}, which means that $ \varphi = \Gamma(\psi_1) $ solves the ODE $ \dot{\varphi} = \bar{b}(\varphi) + \bar{\sigma}(\varphi) \dot{\psi}_1 $.
\end{theorem}

\begin{proof}	
	We know that $  (\sqrt{\varepsilon} B, V^{\varepsilon}) $ satisfies the LDP	in Theorem~\ref{thm:ldp for (B,V)}.	
	The mapping $ (\varphi_1, \varphi_2) \mapsto (\varphi_1, \hat{\varphi}_2) $ of $ C[0, T]^2 $ into itself is continuous due to Lemma~\ref{lem:f to f hat is continuous}. Hence, we can use the contraction principle, which gives
	\begin{align*}
		\tilde{I}(\psi_1, \psi_2)
		= \inf_{\{ (\varphi_1, \varphi_2) \in C[0, T]^2 :\, (\psi_1, \psi_2) = (\varphi_1, \hat{\varphi}_2) \}}
		I(\varphi_1, \varphi_2) =  \inf_{\hat{\varphi}_2 = \psi_2} I(\psi_1, \varphi_2).
	\end{align*}
	The necessary condition under which we have $ I(\psi_1, \varphi_2) < \infty $ is $ \dot{\psi}_1 = \frac{\dot{\varphi}_2 - \bar{b}(\varphi_2)}{\bar{\sigma}(\varphi_2)} $ (see Corollary~\ref{cor: rate function for phi_1 absolutely continuous}). 
\end{proof}
Since $B$ and $W$ are independent, the following result is an immediate consequence
of Theorem~\ref{thm:ldp for epsB, Vhat} and Schilder's theorem.
\begin{corollary}\label{cor:ldp for W,B,V hat}
	\begin{enumerate}[label={\normalfont (\roman*)},ref={\normalfont \roman*}]
		
		\item The family  $  (\sqrt{\varepsilon} W_T, \sqrt{\varepsilon} B, \hat{V}^{\varepsilon}) $ satisfies an LDP with speed $ \varepsilon^{-1} $ and rate function
		\begin{align}
			\hat{I}\big(y, \psi_1, \mathcal{K}(U \circ \Gamma(\psi_1))\big)
			= \frac{T}{2} y^2
			+ \frac{1}{2} \int_{0}^{T} \dot{\psi}_1^2\, dt,
		\end{align}
		for $ y \in \mathbb{R} $ and $ \psi_1 \in H_0^1[0, T],$ if all the expressions are finite, and $ \hat{I}(y, \psi_1, \psi_2) = \infty $ otherwise.
		
		\item The family of processes $  (\sqrt{\varepsilon} W, \sqrt{\varepsilon} B, \hat{V}^{\varepsilon}) $  satisfies an LDP with speed $ \varepsilon^{-1} $ and rate function
		\begin{align}
			\hat{I}\big(\psi_0, \psi_1, \mathcal{K}(U \circ \Gamma(\psi_1))\big)
			= \frac{1}{2} \int_{0}^{T} \dot{\psi}_0(t)^2\, dt
			+ \frac{1}{2} \int_{0}^{T} \dot{\psi}_1^2\, dt,
		\end{align}
		for $ \psi_0, \psi_1 \in H_0^1[0, T],$ if all the expressions are finite, and $ \hat{I}(\psi_0, \psi_1, \psi_2) = \infty $ otherwise.
	\end{enumerate}
\end{corollary}

\section{Proof of the LDP for the log-price}
\label{sec:small-noise proofs}

\subsection{Proof of Theorem~\ref{thm:ldp for log-price (with W_T)} (one-dimensional LDP)}

It is clear that the one-dimensional LDP in Theorem~\ref{thm:ldp for log-price (with W_T)}
is a special case of the sample path LDP in Theorem~\ref{thm:ldp for log-price}.
For the reader's convenience, though, it seemed better to us to first
prove Theorem~\ref{thm:ldp for log-price (with W_T)}, and then refer to
some parts of this proof in the proof of Theorem~\ref{thm:ldp for log-price} below.
We build on some ideas of~\cite{Gu18}. To match the notation there, we note that $ \varepsilon^{H} \hat{B} $ from~\cite{Gu18} corresponds to our process $ \hat{V}^{\varepsilon} $ as defined in \eqref{eq:V hat epsilon}. In the original proof of \cite{Gu18} the author first supposes $ T = 1 $. Here, for convenience, we immediately allow a general $ T > 0 $.
By the following lemma, it suffices to prove an LDP for the driftless process
\begin{align}\label{eq:X hat}
	d\hat{X}^{\varepsilon}_t
	= \sqrt{\varepsilon} \sigma(\hat{V}^{\varepsilon}_t)(\bar{\rho}\, dW_t + \rho\, dB_t),
	\quad 0 \leq t \leq T.
\end{align}
\begin{lemma}\label{lem:omit drift scalar}
	The families $ (X^\varepsilon_T)_{\varepsilon>0} $ and $ (\hat{X}^{\varepsilon}_T)_{\varepsilon>0} $ are exponentially equivalent, i.e.\ for every $ \delta > 0 $, the following equality holds:
	\begin{align}\label{eq:exponential equivalence drift on state space}
		\limsup_{\varepsilon \searrow 0} \varepsilon \log P(|X_T^{\varepsilon} - \hat{X}_T^{\varepsilon}| > \delta)
		= - \infty.
	\end{align}
\end{lemma}

\begin{proof}
	By the same reasoning as in Section~5 of \cite{Gu18}, there is a strictly increasing continuous function $ \eta : [0, \infty) \to [0, \infty) $ with $ \lim\limits_{u \nearrow \infty} \eta(u) = 
	\infty $ and $ \bar{\sigma}(u)^2 \leq \eta(u) $ for all $ u \in \mathbb{R} $. Let $ \eta^{-1} : [0, \infty) \to [0, \infty) $ be the inverse function. Replacing $ \sqrt{\varepsilon}\hat{B} $  in \cite{Gu18} by $ \hat{V}^{\varepsilon} $, we get the estimate
	\begin{align}\begin{split}\label{eq:omit drift estimate}
		P(|X_T^{\varepsilon} - \hat{X}_T^{\varepsilon}| > \delta)
		& = P\Big(\frac{1}{2} \varepsilon \int_{0}^{T} \sigma(\hat{V}^{\varepsilon}_s)^2\, ds > \delta\Big)
		\leq P\Big(\frac{1}{2} \varepsilon \int_{0}^{T} \eta(\hat{V}^{\varepsilon}_s)\, ds > \delta\Big)\\
		& \leq P\Big(\frac{1}{2} \varepsilon \int_{0}^{T} \eta(\sup_{0 \leq t \leq T} |\hat{V}^{\varepsilon}_t|)\, ds > \delta\Big)
		= P\Big(\frac{1}{2} \varepsilon T \eta(\sup_{0 \leq t \leq T} |\hat{V}^{\varepsilon}_t|) > \delta\Big)\\
		& = P\Big(\eta(\sup_{0 \leq t \leq T} |\hat{V}^{\varepsilon}_t|) > \frac{2 \delta}{\varepsilon T}\Big)
		= P\Big(\sup_{0 \leq t \leq T} |\hat{V}^{\varepsilon}_t| > \eta^{-1}(\frac{2 \delta}{\varepsilon T})\Big)\\
		& \leq \exp\Big({-\frac{\varepsilon^{-1}}{2}} J(A)\Big),
	\end{split}\end{align}
	where $ J $ is the rate function of $ \sup_{0 \leq t \leq T} |\hat{V}^{\varepsilon}_t| $,
	and $A=(\eta^{-1}(\frac{2 \delta}{\varepsilon T}), \infty)$. Since $ J $ is a good rate function, we know that $ J(x, \infty) \nearrow \infty $ as $ x \nearrow \infty $, so we get \eqref{eq:exponential equivalence drift on state space}. 
\end{proof}

We will next reason as in \cite{Gu18}, p.~1121, using the
LDP for  $ (\sqrt{\varepsilon} W_T, \sqrt{\varepsilon} B, \hat{V}^{\varepsilon}) $ in Corollary~\ref{cor:ldp for W,B,V hat}.
Analogously to~\cite{Gu18}, we define the functional $ \Phi $ on the space $ M = \mathbb{R} \times C[0, T]^2 $ by
\begin{align}\label{eq:def phi}
	\Phi(y, f, g)
	= \bar{\rho} \Big(\int_{0}^{T} \sigma(g(s))^2\, ds\Big)^{1 / 2} y
	+ \rho \int_{0}^{T} \sigma(g(s)) \dot{f}(s)\, ds,
\end{align}
if $ (f, g) = (f, \check{f}) $ with $ f \in H_0^1[0, T],$ and $ \Phi(y, f, g) = 0 $ otherwise
(recall the definition~\eqref{eq:f check (within def)}).
Further, for any integer $ m \geq 1 $, define a functional on $ M $ by
\begin{align}\label{eq:def phi_m}
	\Phi_m(y, h, l)
	= \bar{\rho} \Big(\int_{0}^{T} \sigma(l(s))^2\, ds\Big)^{1 / 2} y
	+ \rho \sum_{k = 0}^{m - 1} \sigma(l(t_k)) \big(h(t_{k + 1}) - h(t_k)\big),
\end{align}
where $ t_k := \frac{kT}{m} $ for $ k \in \{ 0, \ldots, m \} $.
The following approximation property is the key to applying the extended contraction
principle (see (4.2.24) in \cite{DeZe98}).
\begin{lemma}\label{lem:(4.2.24) for dembo zeitouni}
	For every $ \alpha > 0 $,
	\begin{align}\label{eq:(4.2.24) for dembo zeitouni}
		\limsup_{m \to \infty} \sup_{\{ f \in H_0^1[0, T] : \frac{T}{2} y^2 + \frac{1}{2} \int_{0}^{T} \dot{f}(s)^2\, ds \leq \alpha \}} |\Phi(y, f, \check{f}) - \Phi_{m}(y, f, \check{f})|
		= 0.
	\end{align}
\end{lemma}

\begin{proof}
	The proof is similar to that of Lemma~21 in \cite{Gu18}. We just need to change the range of the integrals and suprema to $ [0, T] $ instead of $ [0, 1] $. Hence, the grid points for $ h_m $ are $ t_k := \frac{Tk}{m} $ for $ k \in \{0, \ldots, m\} $, like in \eqref{eq:def phi_m}. We use a different integral operator
	than~\cite{Gu18}, and so we have to show that 
	the set $ E_{\beta} = \{\check{f} : f \in D_{\beta} \} $ is precompact in $ C[0, T] $ for $ D_{\beta} = \{ f \in H_0^1[0, T] : \int_{0}^{T} \dot{f}(s)^2\, ds < \beta \} $.
	For $ f \in D_\beta $, we have $ \dot{f} \in L^2[0, T] $ and therefore can use Eq.~(16) of \cite{ChFi14} to estimate the solution of the ODE
	\begin{align*}
		v
		= v_0
		+ \int_{0}^{\cdot} \bar{b}(v(s))\, ds
		+ \int_{0}^{\cdot} \bar{\sigma}(v(s)) \dot{f}(s)\, ds
	\end{align*}
	as follows:
	\begin{align*}
		\sup_{0 \leq s \leq T} |v(s)|^2
		\leq \big(3 |v_0|^2 + 6 \mu^2 T^2 + 6 \mu^2 T \| \dot{f} \|_2^2\big) e^{ 6 \mu^2 T (T + \| \dot{f} \|_2^2) }
		=: C_\beta^2.
	\end{align*}
	Here, $\mu$ comes from the sub-linear growth condition for the coefficient functions of the diffusion equation for $ V $ in Assumption~\ref{ass:assumptions for cir type diffusion}.
	Since the continuous function~$U$ is bounded on the interval $[-C_\beta,C_\beta]$,
	\begin{align}\label{eq:bounded set}
		\{ U \circ v : f \in D_{\beta},\, \dot{v} = \bar{b}(v) + \bar{\sigma}(v) \dot{f} \}
	\end{align}
	is a bounded subset of $ C[0, T] $. The compact operator $ \mathcal{K} $, as defined in \eqref{eq:op K}, maps the set in~\eqref{eq:bounded set} to a precompact set in $ C[0, T] $. So we can conclude that $ E_{\beta} $  is precompact. After that, the proof continues like in \cite{Gu18}.
\end{proof}


\begin{definition}
	Let $ t \in [0, T] $ be fixed. Consider the grid $ t_k := T \frac{k}{m} $ for $ k \in \{0, \ldots, m\} $. There is a $ k $ such that $ t \in [t_k, t_{k + 1}) $. Denote by $ \Xi(t) $ the left end of the  previous interval. Explicitly, we put
	\begin{align}
		\Xi(t) := \frac{T}{m} [\frac{mt}{T}],
	\end{align}
	where $ [a] $ stands for the integer part of the number $ a \in \mathbb{R} $.
	For $ T = 1 $, this reduces to $ \Xi(t) =  \frac{[mt]}{m} $.
\end{definition}

We will next prove that  $  \Phi_m(\sqrt{\varepsilon} W_T, \sqrt{\varepsilon} B, \hat{V}^{\varepsilon}) $ is an exponentially good approximation as $ m \nearrow \infty $ to  $  (\sqrt{\varepsilon} W_T, \sqrt{\varepsilon} B, \hat{V}^{\varepsilon}) $. We start with an auxiliary result.

\begin{lemma}\label{le:lemma 27 from Gu18_arxiv}
	For every $ y > 0, $
	\begin{align}\label{eq:lemma 27 from Gu18_arxiv} 
		\limsup_{ m \to \infty } \limsup_{\varepsilon \searrow 0} \varepsilon \log P\Big(\sup_{t \in [0, T]} |\hat{V}^{\varepsilon}_t - \hat{V}^{\varepsilon}_{\Xi(t)}| > y \Big)
		= - \infty.
	\end{align}
\end{lemma}

\begin{proof}
	This corresponds to Lemma~23 in \cite{Gu18}, but we need to adjust some
	estimates in the proof, since we do \emph{not} have Gaussianity in our setting. 
	As in \cite{Gu18} we use
	\begin{align}\label{eq:lemma 27 1st estimate}
		P\Big(\sup_{t \in [0, T]} |\hat{V}^{\varepsilon}_t - \hat{V}^{\varepsilon}_{\Xi(t)}| > y \Big)
		\leq P\bigg(\sup_{\substack{t_1, t_2 \in [0, T]\\ |t_2-t_1|\leq T / m}} |\hat{V}^{\varepsilon}_{t_2} - \hat{V}^{\varepsilon}_{t_1}| > y\bigg).
	\end{align}
	Then, for $ |s - t| \leq T / m $, we have
	\begin{align*}\begin{split}
		|\hat{V}^{\varepsilon}_t - \hat{V}^{\varepsilon}_s|
		&= \Big|\int_{0}^{T} \big(K(t, v) - K(s, v)\big) U(V_v^{\varepsilon})\, dv\Big|\\
		&\leq \sqrt{M(\frac{T}{m})} \sup_{v \in [0, T]} |U(V_v^{\varepsilon})|\\
		&\leq \Big(\frac{c T}{m}\Big)^{r / 2} \sup_{v \in [0, T]} |U(V_v^{\varepsilon})|,
	\end{split}\end{align*}
	where $ M $ is the modulus of continuity of the kernel function in  Assumption~\ref{ass:kernel definition}. We know that $  V^{\varepsilon} $ satisfies an LDP, by Theorem~\ref{thm:ldp for (B,V)}. Using this, we can estimate
	\begin{align*}
		P\Big(\sup_{t \in [0, T]} |\hat{V}^{\varepsilon}_t - \hat{V}^{\varepsilon}_{\Xi(t)}| > y \Big) &\leq
		 P\Big(\sup_{s \in [0, T]} |U(V_s^{\varepsilon})| > y c^{-r/2} T^{-r/2} m^{r/2} \Big) \\
		&\leq \exp\Big({-\frac{\varepsilon^{-1}}{2}} \cdot J\big((y (\frac{m}{cT} )^{\frac{r}{2}}, \infty)\big)\Big),
	\end{align*}
	for $ \varepsilon $ small enough. Here, $ J $ is the good rate function corresponding to $ \sup_{s \in [0, T]} |U(V_s^{\varepsilon})|, $ which satisfies an LDP, as seen from applying the contraction principle to the continuous mapping $ f \mapsto \sup_{s \in [0, T]} |U(f(s))| $.
	From this, we can write
	\begin{align}
		\limsup_{\varepsilon \searrow 0} \varepsilon \log P\Big(\sup_{t \in [0, T]} |\hat{V}^{\varepsilon}_t - \hat{V}^{\varepsilon}_{\Xi(t)}| > y\Big)
		\leq - \frac{1}{2} J\Big(\big(y \big(\frac{m}{cT} \big)^{\frac{r}{2}}, \infty\big)\Big).
	\end{align}
	Since $ J $ has compact level sets, the term on the right-hand side explodes
	for $ m \nearrow \infty $.
\end{proof}

Next, we show that the discretization functionals~$\Phi_m$
yield an exponentially good approximation.

\begin{lemma}\label{lem:exponential equivalence for scalar case}
	For every $ \delta > 0 $,
	\begin{align}
		\lim\limits_{m \to \infty} \limsup_{\varepsilon \searrow 0} \varepsilon \log P(|\Phi(\sqrt{\varepsilon}W_T, \sqrt{\varepsilon} B, \hat{V}^{\varepsilon}) - \Phi_{m}(\sqrt{\varepsilon}W_T, \sqrt{\varepsilon} B, \hat{V}^{\varepsilon})| > \delta) = - \infty.	
	\end{align}
\end{lemma}

\begin{proof}
	This lemma corresponds to Lemma~22 in \cite{Gu18}. 
	As in the proof of that lemma, it suffices to show
	\begin{align}\label{eq:stronger condition for exponential equivalence}
		\lim\limits_{m \to \infty} \limsup_{\varepsilon \searrow 0} \varepsilon \log P\bigg(\sqrt{\varepsilon} |\rho| \sup_{t \in [0, T]} \Big|\int_{0}^{t} \sigma_s^{(m)}\, dB_s\Big| > \delta\bigg)
		= -\infty,
	\end{align}
	where $ \sigma_t^{(m)} = \sigma(\hat{V}^{\varepsilon}_t) - \sigma(\hat{V}^{\varepsilon}_{\Xi(t)}) $.	We have to redefine $ \xi_{\eta}^{(m)} $ in order to take a general $ T > 0 $ into account:
	\begin{align*}
		\xi_{\eta}^{(m)}
		= \inf \Big\{ t \in [0, T] : \frac{\eta}{q(\eta)} |\hat{V}^{\varepsilon}| + |\hat{V}^{\varepsilon}_t - \hat{V}^{\varepsilon}_{\Xi(t)}| > \eta \Big\} \wedge T.
	\end{align*}
	Note that we use the convention $ \inf \emptyset = \infty $ here.
	The equations (55)--(65) in \cite{Gu18}  remain the same, except that we replace $ \varepsilon^H \hat{B} $ by $ \hat{V}^{\varepsilon} $ and use our redefined versions of $ \sigma^{(m)} $ and $ \xi_{\eta}^{(m)} $. Thus,
	formula~(65) in \cite{Gu18} can be applied. The estimates~(66) and (67) have to be replaced by
	\begin{align*}
		P\bigg(\sqrt{\varepsilon} |\rho| \sup_{t \in [0, T]} \Big|\int_{0}^{t} \sigma_s^{(m)}\, dB_s\Big| > \delta\bigg)
		\leq P(\xi_{\eta}^{(m)} < T)
		+ P\bigg(\sqrt{\varepsilon} |\rho| \sup_{t \in [0, \xi_{\eta}^{(m)}]} \Big|\int_{0}^{t} \sigma_s^{(m)}\, dB_s\Big| > \delta\bigg)
	\end{align*}
	and
	\begin{align}\begin{split}
		P(\xi_{\eta}^{(m)} < T)
		& \leq P\Big(\sup_{t \in [0, T]} \big(\frac{\eta}{q(\eta)} |\hat{V}_t^{\varepsilon}| + |\hat{V}_t^{\varepsilon} - \hat{V}^{\varepsilon}_{\Xi(t)}|\big) > \eta \Big)\\
		& \leq P\Big(\sup_{t \in [0, T]} |\hat{V}^{\varepsilon}_t| > \frac{q(\eta)}{2}\Big)
		+ P\Big(\sup_{t \in [0, T]} |\hat{V}^{\varepsilon}_t - \hat{V}^{\varepsilon}_{\Xi(t)}| > \frac{\eta}{2}\Big).
	\end{split}\end{align}
	Using Lemma~\ref{le:lemma 27 from Gu18_arxiv}, we can handle the second term, and so it remains to find an appropriate estimate for the first term. Here we need to adapt the reasoning
	in \cite{Gu18} because of the lack of Gaussianity.
	By the LDP for $  \hat{V}^{\varepsilon} $ and the contraction principle applied to the mapping $ f \mapsto \sup_{t \in [0, T]} |f(t)| ,$ we get
	\begin{align}
		P\Big(\sup_{t \in [0, T]} |\hat{V}^{\varepsilon}_t| > \frac{q(\eta)}{2}\Big)
		\leq \exp\Big({-\frac{\varepsilon^{-1}}{2}} \cdot  I_{\sup}\big((\tfrac12 q(\eta), \infty)\big)\Big),
	\end{align}
	for $ \varepsilon > 0 $ small enough, where $ I_{\sup} $ is the rate function of  $  \sup_{t \in [0, T]} |\hat{V}^{\varepsilon}_t| $.
	Note that $ q(\eta) \nearrow \infty $ for $ \eta \searrow 0 $. So, we get
	\begin{align}\label{eq:estimate for eta term 2}
		\limsup_{\eta \searrow 0} \limsup_{\varepsilon \searrow 0} \varepsilon \log P\Big(\sup_{t \in [0, T]} |\hat{V}_t^{\varepsilon}| > \frac{q(\eta)}{2}\Big)
		= - \infty.
	\end{align}
	Using \eqref{eq:lemma 27 from Gu18_arxiv} and \eqref{eq:estimate for eta term 2}, we get (73) and (74) of \cite{Gu18}. Finally, we can complete the proof  as in \cite{Gu18}.
\end{proof}

Let us continue the proof of Theorem~\ref{thm:ldp for log-price (with W_T)}. Lemma~\ref{lem:(4.2.24) for dembo zeitouni} states that condition~(4.2.24) in \cite{DeZe98} is satisfied. Furthermore, due to Lemma~\ref{lem:exponential equivalence for scalar case}, we know that $ \Phi_m(\sqrt{\varepsilon} W_T, \sqrt{\varepsilon} B, \hat{V}^{\varepsilon}) $ is an exponentially good approximation of $ \Phi(\sqrt{\varepsilon} W_T, \sqrt{\varepsilon} B, \hat{V}^{\varepsilon}) $ as $ m \nearrow \infty $. Hence, we can use the extended contraction principle (Theorem~4.2.23 in \cite{DeZe98}), and get that $ \hat{X}^{\varepsilon}_T $ satisfies an LDP with good rate function $ I $ and speed $ \varepsilon^{-1} $.
We know from Lemma~\ref{lem:omit drift scalar} that $ \hat{X}^{\varepsilon}_T $ and $ X^{\varepsilon}_T $ are exponentially equivalent, and so we finally arrive at  Theorem~\ref{thm:ldp for log-price (with W_T)}.

According to the extended contraction principle, we have
\begin{align*}
	I_T(y)
	= \inf \big\{ \hat{I}(x, f, g) : y = \Phi(x, f, g) \big\}.
\end{align*}
The rate function $ \hat{I} $ is only finite for
\begin{align*}
	\hat{I}\big(y, f, \mathcal{K}(U \circ \Gamma(f))\big)
	= \frac{T}{2} y^2
	+ \frac{1}{2} \langle \dot{f}, \dot{f} \rangle.
\end{align*}
Note that  $ \Gamma $ is  the one-dimensional solution map that takes $ f $ to the solution of the ODE $ \dot{v} = \bar{b}(v) + \bar{\sigma}(v) \dot{f} $, $ v(0) = v_0 $.
Recall that the function $ \Phi $ can be written as
\begin{align*}
	\Phi(y, f, g)
	= \bar{\rho} \sqrt{\langle \sigma(g)^2, 1 \rangle} y
	+ \rho \langle \sigma(g), \dot{f} \rangle.
\end{align*}
Hence, if $ x = \Phi(y, f, g) $, then
\begin{align*}
	y
	= \frac{x - \rho \langle \sigma(g), \dot{f} \rangle}{\bar{\rho} \sqrt{\langle \sigma(g)^2, 1\rangle}}.
\end{align*}
Inserting this into the rate function obtained through the contraction principle, we get
\begin{align}\begin{split}
	I_T(y)
	& = \inf \big\{ \hat{I}(x, f, g) : y = \Phi(x, f, g), ~f \in H_0^1, ~g = \mathcal{K}(U \circ \Gamma(f)) \big\}\\
	& = \inf \Big\{ \frac{T}{2} y^2 + \frac{1}{2} \langle \dot{f}, \dot{f} \rangle : y = \frac{x - \rho \langle \sigma(\mathcal{K}(U \circ \Gamma(f))), \dot{f} \rangle}{\bar{\rho} \sqrt{\langle \sigma(\mathcal{K}(U \circ \Gamma(f)))^2, 1\rangle}}, f \in H_0^1 \Big\}\\
	& = \inf_{f \in H_0^1} \Big\{ \frac{T}{2} \Big(\frac{x - \rho \langle \sigma(\mathcal{K}(U \circ \Gamma(f))), \dot{f} \rangle}{\bar{\rho} \sqrt{\langle \sigma(\mathcal{K}(U \circ \Gamma(f)))^2, 1\rangle}}\Big)^2 + \frac{1}{2} \langle \dot{f}, \dot{f} \rangle \Big\}.
\end{split}\end{align}

\subsection{Proof of Theorem~\ref{thm:ldp for log-price} (a sample path LDP)}

We adapt the arguments on pp.~8--11 in~\cite{Gu19}. 
As in the preceding section, our starting point is that we already have an LDP for $ (\sqrt{\varepsilon} W, \sqrt{\varepsilon} B, \hat{V}^{\varepsilon}) $, see Corollary~\ref{cor:ldp for W,B,V hat}.
We redefine the functions $ \Phi $ and $ \Phi_m $  so that they map  $ C[0, T]^3 $ to $ C[0, T] $. For $ l \in H_0^1[0, T] $ and $ (f, g) \in C[0, T]^2 $ such that $ f \in H_0^1[0, T] $ and $ g = \check{f} $,
\begin{align}
	\Phi(l, f, g)(t)
	= \bar{\rho} \int_{0}^{t} \sigma(\check{f}(s)) \dot{l}(s)\, ds
	+ \rho \int_{0}^{t} \sigma(\check{f}(s)) \dot{f}(s)\, ds,
	\quad 0 \leq t \leq T.
\end{align}
In addition, for all the remaining triples $ (l, f, g) $, we set $ \Phi(l, f, g)(t) = 0 $ for all $ t \in [0, T] $.
By the following lemma, we can remove the drift term.

\begin{lemma}\label{lem:omit drift path space}
	The families of processes $  X^\varepsilon $ and $  \hat{X}^{\varepsilon} $ are exponentially equivalent, i.e.\ for every $ \delta > 0 $, 
	the following equality holds:
	\begin{align}\label{eq:exponential equivalence drift on path space}
		\limsup_{\varepsilon \searrow 0} \varepsilon \log P(\|X^{\varepsilon} - \hat{X}^{\varepsilon}\|_{C[0, T]} > \delta)
		= - \infty.
	\end{align}
	Here, $ \hat{X}^{\varepsilon} $ is defined in~\eqref{eq:X hat}.
\end{lemma}

\begin{proof}
  By taking into account the proof of Lemma~\ref{lem:omit drift scalar}, we see that just
  one additional estimate is needed, namely
\begin{align*}
	\| X^{\varepsilon} - \hat{X}^{\varepsilon} \|_{C[0, T]}
	= \sup_{0 \leq t \leq T} | X^{\varepsilon}_t - \hat{X}^{\varepsilon}_t |
	\leq \frac{1}{2} \varepsilon T \eta\big(\sup_{0 \leq t \leq T} |\hat{V}^{\varepsilon}_t|\big).
\end{align*}
Then we directly get
\begin{align*}
	P(\| X^{\varepsilon} - \hat{X}^{\varepsilon} \| > \delta)
	\leq P\Big(\frac{1}{2} \varepsilon T \eta\big(\sup_{0 \leq t \leq T} |\hat{V}^{\varepsilon}_t|\big) > \delta\Big)
	= P\Big(\sup_{0 \leq t \leq T} |\hat{V}^{\varepsilon}_t| > \eta^{-1}\big(\frac{2 \delta}{\varepsilon T}\big) \Big),
\end{align*}
which is exactly the same expression as in the proof of \eqref{eq:exponential equivalence drift on state space}.
\end{proof}

The sequence of functionals $ (\Phi_m)_{m \geq 1} $ from $ C[0, T]^3 $ to $ C[0, T] $ is given for $ (r, h, l) \in C[0, T]^3 $ and $ t \in [0, T] $ by
\begin{align}\begin{split}
	\Phi_m(r, h, l)(t)
	= \bar{\rho} \bigg(\sum_{k = 0}^{[\frac{mt}{T} - 1]} \sigma(l(t_k)) [r(t_{k + 1}) - r(t_k)]
	+ \sigma\big(l(\Xi(t))\big) \big[r(t) - r(\Xi(t))\big]\bigg)\\
	+ \rho \bigg(\sum_{k = 0}^{[\frac{mt}{T} - 1]} \sigma(l(t_k)) [h(t_{k + 1}) - h(t_k)] + \sigma\big(l(\Xi(t))\big) \big[h(t) - h(\Xi(t))\big]\bigg).
\end{split}\end{align}
It is not hard to see that for every $ m \geq 1 $, the mapping $ \Phi_m $ is continuous.

\begin{lemma}\label{lem:dembo zeitoungi (4.2.24) for path space}
	For every $ \zeta > 0 $ and $ y > 0 $,
	\begin{align}
		\limsup_{m \nearrow \infty} \sup_{\{(r, f) \in H_0^1[0, T]^2 :\,  \frac{1}{2} \int_{0}^{T} \dot{r}(s)\, ds + \frac{1}{2} \int_{0}^{T} \dot{f}(s)\, ds \leq \zeta\}} \| \Phi(r, f, \check{f}) - \Phi_m(r, f, \check{f}) \|_{C[0, T]^2}
		= 0.
	\end{align}
\end{lemma}

\begin{proof}
	Lemma~\ref{lem:dembo zeitoungi (4.2.24) for path space}
	can be obtained from the proofs of Lemma~\ref{lem:(4.2.24) for dembo zeitouni},  Lemma~21 in \cite{Gu18} and Lemma~2.13 in \cite{Gu19}. The only difference here is, that the supremum is taken over two functions from $ D_\eta  = \{ w \in H_0^1[0, T] : \int_{0}^{T} \dot{w}^2\, ds \leq \eta \} $. By the uniform bound in the proof of Lemma~21 of \cite{Gu18}, this is actually irrelevant.
\end{proof}

Next, we will show  that the family $ \Phi_m(\sqrt{\varepsilon} W, \sqrt{\varepsilon} B, \hat{V}^{\varepsilon}) $ is an exponentially good approximation for $ \Phi(\sqrt{\varepsilon} W, \sqrt{\varepsilon} B, \hat{V}^{\varepsilon}) $, as $ m \nearrow \infty $.

\begin{lemma}\label{lem:exponentially good approximation on path space}
	For every $ \delta > 0 $
	\begin{align}
		\lim\limits_{m \to \infty} \limsup_{\varepsilon \searrow 0} \varepsilon \log P(\| \Phi(\sqrt{\varepsilon}W, \sqrt{\varepsilon} B, \hat{V}^{\varepsilon}) - \Phi_m(\sqrt{\varepsilon}W, \sqrt{\varepsilon} B, \hat{V}^{\varepsilon}) \|_{C[0, T]} > \delta) = -\infty.
	\end{align}
\end{lemma}

\begin{proof}
	In the proof of
	Lemma~\ref{lem:exponential equivalence for scalar case},
	the estimate \eqref{eq:stronger condition for exponential equivalence} was formulated stronger than needed.
	We can directly use this to show (2.13) of \cite{Gu19}. We can also get (2.14) of \cite{Gu19} this way. The ingredients of (55)--(65) in \cite{Gu18} do in fact depend on the Brownian motion $ B $ via the process $ \hat{V}^{\varepsilon} $. However, the reasoning for the estimate
	\begin{align}
		P\bigg(\sup_{t \in [0, \xi_\eta^{(m)}]} \varepsilon^H \Big| \int_{0}^{t} \sigma_s^{(m)}\, dB_s \Big| > \delta\bigg)
		\leq \exp \Big(-\frac{\delta^2}{2 \varepsilon^{2 H} L(q(\eta))^2 \omega(\eta)^2}\Big)
	\end{align}
	in \cite{Gu18} stays the same if we replace the driving Brownian motion $ B $ by $ W $.
	The rest of the proof from here on is essentially the same as in the proof of Theorem~2.9 in \cite{Gu19}.
\end{proof}

Just as in the preceding section, we combine Lemmas~\ref{lem:omit drift path space}--\ref{lem:exponentially good approximation on path space} to see that
Theorem~\ref{thm:ldp for log-price} follows from the extended contraction principle. 
 We have
\begin{align*}
	Q(g) 
	= \inf \{ \hat{I}(\psi_0, \psi_1, \psi_1) : g = \Phi(\psi_0, \psi_1, \psi_2) \}.
\end{align*}
The rate function $ \hat{I} $ is only finite for
\begin{align*}
	\hat{I}(\psi_0, \psi_1, \psi_2)
	= \frac{1}{2} \langle \dot{\psi_0}, \dot{\psi_0} \rangle
	+ \frac{1}{2} \langle \dot{f}, \dot{f} \rangle,
\end{align*}
where $ \psi_1 = f $ and $ \psi_2 = \mathcal{K}(U \circ \Gamma(f)) $ for some $ f \in H_0^1[0, T] $.
Recall that the function $ \Phi $  is given by
\begin{align*}
	\Phi(l, f, g)(t)
	= \bar{\rho} \int_{0}^{t} \sigma(g(s)) \dot{l}(s)\, ds
	+ \rho \int_{0}^{t} \sigma(g(s)) \dot{f}(s)\, ds,
\end{align*}
hence we can write
\begin{align*}
	\dot{l}
	= \frac{\partial_t (\Phi(l, f, g)) - \rho \sigma(g) \dot{f}}{\bar{\rho} \sigma(g)}.
\end{align*}
Finally, we get the rate function as follows:
\begin{align}\begin{split}
	Q(g)
	& = \inf \{ \hat{I}(\psi_0, \psi_1, \psi_2) : g = \Phi(\psi_0, \psi_1, \psi_2) \}\\
	& = \inf \Big\{ \frac{1}{2} \langle \dot{\psi_0}, \dot{\psi_0} \rangle + \frac{1}{2} \langle \dot{f}, \dot{f} \rangle : f \in H_0^1, ~\psi_1 = f, ~\psi_2 = \mathcal{K}(U \circ \Gamma(f)),\\
	& ~~~~~~~~~~~~~~~~~\dot{\psi}_0 = \frac{\partial_t (\Phi(\psi_0, \psi_1, \psi_2)) - \rho \sigma(\psi_2) \dot{\psi}_1}{\bar{\rho} \sigma(\psi_2)}, ~g = \Phi(\psi_0, \psi_1, \psi_2) \Big\}\\
	& = \inf \bigg\{ \frac{1}{2} \langle \dot{\psi}_0, \dot{\psi}_0 \rangle + \frac{1}{2} \langle \dot{f}, \dot{f} \rangle : f \in H_0^1, ~\dot{\psi}_0 = \frac{\dot{g} - \rho \sigma(\mathcal{K}(U \circ \Gamma(f))) \dot{f}}{\bar{\rho} \sigma(\mathcal{K}(U \circ \Gamma(f)))} \bigg\}\\
	& = \inf_{f \in H_0^1} \bigg\{ \frac{1}{2} \int_{0}^{T} \Big(\frac{\dot{g}(t) - \rho \sigma(\mathcal{K}(U \circ \Gamma(f))(t)) \dot{f}(t)}{\bar{\rho} \sigma(\mathcal{K}(U \circ \Gamma(f))(t))}\Big)^2\, dt + \frac{1}{2} \int_{0}^{T} |\dot{f}(t)|^2\, dt \bigg\}.
\end{split}\end{align}

\section{Large strike asymptotics}\label{sec:strike}

Under suitable scaling assumptions, large strike asymptotics of call prices are a natural
consequence of our small-noise LDP. To achieve a convenient scaling w.r.t.\ space, we assume in this 
section that
\[
  \sigma(x) = \sigma_0(1+x^\beta)
\]
for some $\sigma_0>0$ and $\beta \in(0,\tfrac12)$. Furthermore, $V$ is a drift-less
CIR process, i.e.\ $ \bar{\sigma}(x) = \sqrt{x} $ and $ \bar{b} \equiv 0 $,
and we take $U = \operatorname{id}$. We are thus dealing with a fractional
Heston-type model, where some degree of generality is preserved, as~$K$ may be an arbitrary kernel satisfying
Assumption~\ref{ass:kernel definition}.
We note that \emph{small time} asymptotics of this model are not within the scope
of our approach, because the standard transfer involving Brownian scaling leads (for the
fractional kernel)
to a small time regime where log-moneyness \emph{increases} as maturity shrinks,
which is of little practical interest.
Therefore, we consider large-strike approximations instead.
The drift-less log-price is
\begin{align*}
  \hat{X}_T &= \sigma_0(\bar{\rho} W_T + \rho B_T)
  + \sigma_0 \int_{0}^{T} (\hat{V}_t)^\beta\, d (\bar{\rho} W_t + \rho B_t) \\
  &=: \sigma_0(\bar{\rho} W_T + \rho B_T)+ \tilde{X}_T,
\end{align*}
and it is easy to see that the tail of the Gaussian term $\sigma_0(\bar{\rho} W_T + \rho B_T)$
is negligible, as is the passage from the log-price $X_T$ to $\hat{X}_T$.
It is clear from our assumptions that $\varepsilon V\stackrel{d}{=} V^{\varepsilon},$
and thus $\varepsilon \hat{V}\stackrel{d}{=} \hat{V}^{\varepsilon},$ for any $\varepsilon>0$.
Therefore,
\begin{align*}
  \varepsilon^{\beta+1/2}  \tilde{X}_T
  &= \sqrt{\varepsilon} \sigma_0  \int_{0}^{T} (\varepsilon\hat{V}_t)^\beta\, d (\bar{\rho} W_t + \rho B_t) \\
  &\stackrel{d}{=} \sqrt{\varepsilon}\sigma_0  \int_{0}^{T} (\hat{V}_t^\varepsilon)^\beta\, d (\bar{\rho} W_t + \rho B_t). 
\end{align*}
Then, Theorem~\ref{thm:ldp for log-price (with W_T)} implies, for any $c>0$, that
%
%
\[
  P(\varepsilon^{\beta+1/2}  X_T \geq c) = \exp\Big({-\frac{I_T((c,\infty)) }{\varepsilon}}
  (1+o(1)) \Big), \quad \varepsilon \searrow 0.
\]
Writing $k=\varepsilon^{-(\beta+1/2)}$ and $\gamma=(\beta+\tfrac12)^{-1}\in(1,2)$, we obtain
\begin{equation}\label{eq:strike}
   P(X_T \geq k) = \exp\big( {-I_T((1,\infty))}k^\gamma (1+o(1)) \Big), \quad k\nearrow \infty,
\end{equation}
for $c=1$, and replacing $k$ by $ck$ we see that the rate function satisfies the scaling
property
\begin{equation}\label{eq:scaling}
  I_T((c,\infty)) = c^\gamma I_T((1,\infty)), \quad c>0.
\end{equation}
This  easily implies that the rate function is given by
\[
  I_T(c) = c^\gamma I_T(1),\quad c>0.
\]
For the digital call price, \eqref{eq:strike} then yields
\begin{equation}\label{eq:dig}
  P(S_T \geq K) = \exp\big( {-I_T(1)}(\log K)^\gamma (1+o(1)) \big),
  \quad K\nearrow \infty;
\end{equation}
no confusion between the strike~$K$ and the kernel $K(\cdot,\cdot)$ should arise. Note that the choice of
the latter affects the value of $I_T(1)$ in~\eqref{eq:dig}.
Since $\gamma\in(1,2)$, this shows that the stock price $S_T$ has finite moments of all orders $p>0$.
Then Theorem~1.1 in~\cite{BeFr09} shows that call prices have the same logarithmic
large-strike
asymptotics as digital calls, which establishes the following result.
\begin{proposition}
In the model described at the beginning of this section, the call price satisfies
\[
  E[(S_T-K)^+] = \exp\big( {-I_T(1)}(\log K)^\gamma (1+o(1)) \big),
  \quad K\nearrow \infty,
\]
where $\gamma=(\beta+\tfrac12)^{-1}\in(1,2).$
\end{proposition}

\begin{remark}\label{rem:hoelder}
  The paths of the CIR process~$V$ are $(\tfrac12-\delta)$-H\"older continuous
  for any $\delta\in(0,\tfrac12)$ (see Lemma~7.1 in~\cite{BaDe20}).
  If we choose the fractional kernel $K(s,t)=(t-s)^{H-1/2}$, $H\in(0,1),$ in the model considered in the present section, then the paths of $\hat{V}$ are in the
  H\"older space $\mathcal{H}^{H+1-\delta}$. See Definition~1.1.6 (p.~6)
  and Corollary~1.3.1 (p.~56) in~\cite{SaKiMa93}. In particular, since
  $H+1-\delta>1$ for small~$\delta$, the paths
  of~$\hat{V}$ are $C^1$ on $(0,T)$. By modifying the model, using
  $U(x)=|x-V_0|^\kappa$ with $\kappa\in(0,1]$ instead of $U=\mathrm{id},$
  the paths of~$\hat V$ become less smooth, namely $(\tfrac12\kappa+H+\tfrac12-\delta)$-H\"older continuous.
    In addition, if $\sigma(x)=\sigma_0(1+x^{\beta})$, then the volatility
  paths $t\mapsto\sigma_0(1+(\hat{V_t})^\beta)$ are $(\tfrac12\kappa\beta+(H+\tfrac12)\beta-\delta)$-H\"older continuous on $[0,T]$, for any small
  enough $\delta>0$. While this H\"older exponent can be smaller than~$\tfrac12$, the volatility process is \emph{not} rough, because $\sigma(\cdot)$ is smooth
  away from zero, and so ``roughness'' occurs only at time zero. Note
  that in truly rough models, the volatility process is  constructed using stochastic integrals
  $\int_{0}^t K(t,s)dW_s$ or related processes, which is not the case in our setup.
\end{remark}

\section{Second order Taylor expansion of the rate function}
\label{sec:energy expansion}

In order to compute the rate function, a certain variational problem needs to be solved numerically. It might be preferable to use the Taylor expansion
of the rate function instead, if it can be computed in closed form.
 The model from the preceding section is a case in point: By the scaling
property~\eqref{eq:scaling}, we may evaluate the rate function at a small $c>0$ of our choice.
For the special case where $ V = B^2 $ and $ U(x) = x $ or, alternatively, $ V = B $, $ U(x) = x^2 $, $ \bar{\sigma} \equiv 1 $ and $ \bar{b} \equiv 0 $, i.e.\ $ \Gamma \equiv \operatorname{id} $, we
now discuss how to expand the rate function, building on \cite{BaFrGuHoSt19}.

\begin{proposition}\label{prop:rate function second derivative} Let $ U \equiv \operatorname{id} $ and $ V \equiv B^2 $. Furthermore, assume that $ \sigma $ is smooth (at least locally around $ 0 $). Suppose that the rate function $ I $ is also smooth locally around $ 0 $. Then, its Taylor expansion is
	\begin{align}
		\notag I(x)
		& = I(0)
		+ I'(0) x
		+ I''(0) x^2
		+ O(x^3)\\
		\notag & = I''(0) x^2 + O(x^3)\\
		& = \frac{1}{2 \sigma_0^2} x^2 + O(x^3). \label{eq:I Taylor}
	\end{align}
\end{proposition}

\begin{remark}
 Formula~\eqref{eq:I Taylor} gives the second order Taylor expansion.
 However, the ideas in the proof of Proposition~\ref{prop:rate function second derivative} can be used for higher orders. Clearly, the computations for the expansions get much more cumbersome in the latter case.
\end{remark}

\subsection{Proof of Proposition~\ref{prop:rate function second derivative}}


The proof is very similar to the one of Theorem~3.1 in \cite{BaFrGuHoSt19}. In the following, we will outline at which points adjustments are needed.
Note that  for the special we are treating  we have $ U(x) = x^2 $ and $ \Gamma \equiv \operatorname{id} $. To simplify computations in the proof, we use $ T = 1 $.
In Proposition~5.1 of \cite{BaFrGuHoSt19}, there is a representation of the rate function that coincides with ours, except that different integral transforms are used. For our special case, we have
\begin{align}\label{eq:representation via G, F, E of rate function}
	I(x)
	:= \inf_{f \in H_0^1} \Big[\frac{(x - \rho \tilde{G}(f))^2}{2 \bar{\rho}^2 \tilde{F}(f)}
	+ \frac{1}{2} \tilde{E}(f)\Big]
	= \inf_{f \in H_0^1} \mathcal{I}_x(f),
\end{align}
where
\begin{align}
	\tilde{G}(f)
	&:= \int_{0}^{1} \sigma((\mathcal{K} (f^2))(s)) \dot{f}(s)\, ds
	= \langle \sigma(\mathcal{K}(f^2)), \dot{f} \rangle,\\
	\tilde{F}(f)
	&:= \int_{0}^{1} \sigma((\mathcal{K} (f^2))(s))^2\, ds
	= \langle \sigma^2(\mathcal{K} (f^2)), 1 \rangle,\\
	\tilde{E}(f)
	&:= \int_{0}^{1} |\dot{f}(s)|^2\, ds
	= \langle \dot{f}, \dot{f} \rangle.
\end{align}
Recall that $ \mathcal{K} f = \int_{0}^{\cdot} K(\cdot, s) f(s)\, ds $. In \cite{BaFrGuHoSt19} the authors use the same integral transform as used in \cite{Gu18,Gu19}, i.e.\ $ \mathcal{K} \dot{f} $. We have to adjust this to our case of $ \mathcal{K} (f^2) $. Here, $ \mathcal{I}_x $ denotes the functional that needs to be minimized to get the value of the rate function at $ x $.

First, we need to get a representation for the minimizing configuration $ f^x $ of the functional $ \mathcal{I}_x $. This is done like in Proposition~5.2 in \cite{BaFrGuHoSt19}. The corresponding expansions  of the ingredients of the rate function for our setting for $ \delta > 0 $ are
\begin{align}
	\label{eq:expansion E tilde}
	\tilde{E}(f + \delta g)
	&\approx \tilde{E}(f)
	+ 2 \delta \langle \dot{f}, \dot{g} \rangle,\\
	\label{eq:expansion F tilde}
	\tilde{F}(f + \delta g)
	&\approx \tilde{F}(f)
	+ 2 \delta \langle (\sigma^2)'(\mathcal{K}(f^2)), \mathcal{K}(f g) \rangle,\\
	\label{eq:expansion G tilde}
	\tilde{G}(f + \delta g)
	&\approx \tilde{G}(f)
	+ \delta (\langle \sigma(\mathcal{K}(f^2)), \dot{g} \rangle + 2 \langle \sigma'(\mathcal{K}(f^2)), \dot{f} \mathcal{K}(f g) \rangle)
\end{align}
Note, that $ ``\approx" $ is defined in \cite{BaFrGuHoSt19} as
\begin{align}
	A \approx B
	: \Leftrightarrow A = B + o(\delta),
	\quad \delta \searrow 0.
\end{align}


If $ f = f^x $ is a minimizer then $ \delta \mapsto \mathcal{I}_x(f + \delta g) $ has a minimum at $ \delta = 0 $ for all $ g $. Using \eqref{eq:expansion E tilde}, \eqref{eq:expansion F tilde} and \eqref{eq:expansion G tilde} we expand
\begin{align}\begin{split}
	\mathcal{I}_x(f + \delta g)
	& = \frac{(x - \rho \tilde{G}(f + \delta g))^2}{2 \bar{\rho}^2 \tilde{F}(f + \delta g)} + \frac{1}{2} \tilde{E}(f + \delta g)\\
	& \approx \frac{(x - \rho \tilde{G}(f))^2 - 2 \delta \rho (x - \rho \tilde{G}(f)) \big(\langle \sigma(\mathcal{K}(f^2)), \dot{g} \rangle + 2 \langle \sigma'(\mathcal{K}(f^2)), \dot{f} \mathcal{K}(f g) \rangle\big)}{2 \bar{\rho}^2 \tilde{F}(f) \big(1 + \frac{2 \delta}{\tilde{F}(f)} \langle (\sigma^2)'(\mathcal{K}(f^2)), K(f g) \rangle\big)}\\
	&~~~~~ + \frac{1}{2} \tilde{E}(f)
	+ \delta \langle \dot{f}, \dot{g} \rangle\\
	& \approx \frac{(x - \rho \tilde{G}(f))^2 - 2 \delta \rho (x - \rho \tilde{G}(f)) \big(\langle \sigma(\mathcal{K}(f^2)), \dot{g} \rangle + 2 \langle \sigma'(\mathcal{K}(f^2)), \dot{f} \mathcal{K}(f g) \rangle\big)}{2 \bar{\rho}^2 \tilde{F}(f)}\\
	&~~~~~ - \frac{(x - \rho \tilde{G}(f))^2}{2 \bar{\rho}^2 \tilde{F}(f)} \frac{2 \delta}{\tilde{F}(f)} \langle (\sigma^2)'(\mathcal{K}(f^2)), \mathcal{K}(f g) \rangle
	+ \frac{1}{2} \tilde{E}(f)
	+ \delta \langle \dot{f}, \dot{g} \rangle.
\end{split}\end{align}
Now, as a consequence, for $ f = f^x $ and every $ g \in H_0^1[0, 1], $
\begin{align}\begin{split}
	0 = \partial_{\delta} (\mathcal{I}_x(f + \delta g))_{\delta = 0}
	& = - \frac{ 2 \rho (x - \rho \tilde{G}(f))\big(\langle \sigma(\mathcal{K}(f^2)), \dot{g} \rangle + 2 \langle \sigma'(\mathcal{K}(f^2)), \dot{f} \mathcal{K}(f g) \rangle \big)}{2 \bar{\rho}^2 \tilde{F}(f)}\\
	& - \frac{(x - \rho \tilde{G}(f))^2}{2 \bar{\rho}^2 \tilde{F}^2(f)} 2 \langle (\sigma^2)' (\mathcal{K}(f^2), \mathcal{K}(f g) \rangle
	+ \langle \dot{f}, \dot{g} \rangle.
\end{split}\end{align}
We have $ f_0^x = 0 $, for any $ x $. We now test with $ \dot{g} = \ind{[0, t]} $ for a fixed $ t \in [0, 1] $ and obtain
\begin{align}\begin{split}\label{eq:minimizing configuration}
	f_t^x
	& = \frac{\rho (x - \rho \tilde{G}(f^x))\big(\langle \sigma(\mathcal{K}((f^x)^2)), \ind{[0, t]} \rangle + 2  \langle \sigma'(\mathcal{K}((f^x)^2)) , \dot{f^x} \mathcal{K}(f^x \operatorname{id}_{\leq t}) \rangle\big)}{\bar{\rho}^2 \tilde{F}(f^x)}\\
	& + \frac{(x - \rho \tilde{G}(f^x))^2}{2 \bar{\rho}^2 \tilde{F}^2(f^x)} 2 \langle (\sigma^2)'(\mathcal{K}((f^x)^2)), \mathcal{K}(f^x \operatorname{id}_{\leq t}) \rangle,
\end{split}\end{align}
where we write
\begin{align}
	\operatorname{id}_{\leq t} (s)
	= g(s)
	= \int_{0}^{s} \dot{g}(u)\, du
	= \int_{0}^{s} \ind{[0, t]}(u)\, du
	= \int_{0}^{s \wedge t} 1\, du
	= s \wedge t.
\end{align}

Let us recall the ansatz in \cite{BaFrGuHoSt19}. The authors
of \cite{BaFrGuHoSt19} choose for fixed $ x $ the optimizing function $ f^x $ for $ \mathcal{I}_x $, i.e.\ $ f^x = \operatorname{argmin}_{f \in H_0^1} \mathcal{I}_x(f) $. Therefore, the first order condition is $ \mathcal{I}'_x(f^x) = 0 $. The authors of \cite{BaFrGuHoSt19} use the implicit function theorem to show that the minimizing configuration $ f^x $ is a smooth function in $ x $ (locally around $ x = 0 $). As $ \mathcal{I}_x $ is a smooth function, too, this implies the smoothness of $ x \mapsto \mathcal{I}_x(f^x) = I(x) $, at least in a neighborhood of~$ 0 $.
Note that for (26) and Lemma~5.3 in \cite{BaFrGuHoSt19}, the embedding $ \mathcal{K} : H_0^1 \to C $ works, because we have already
established that $ \mathcal{K} (U \circ f) $ is continuous (see Lemma~\ref{lem:f to f hat is continuous}).

In order to apply the implicit function theorem, the authors of \cite{BaFrGuHoSt19} show that the ingredients of the rate function are Fr\'{e}chet differentiable by computing their Gateaux derivative.
This is more complicated in our case, because of the different integral
transform we use.
Therefore we \emph{assume} that the rate function is locally smooth around 0
in Proposition~\ref{prop:rate function second derivative}, and, consequently,
that Lemma~5.6 in \cite{BaFrGuHoSt19} holds.
After establishing that the implicit function theorem can be used, we can proceed as in \cite{BaFrGuHoSt19} up to Theorem~5.12 there.


Next, we will imitate  the computations in Theorem~5.12 of \cite{BaFrGuHoSt19} in order to get the expansion of the minimizing configuration in our setting. In fact, if we just want to obtain the second order expansion of the rate function in our setting for Brownian motion squared, it suffices to find the first order expansion of $ f^x $. Assuming the ansatz
\begin{align}
	f_t^x
	= \alpha_t x
	+ O(x^2),
\end{align}
we get
\begin{align*}
	f_t^x
	& = \alpha_t x
	+ O(x^2),\\
	\dot{f}_t^x
	& = \dot{\alpha}_t x
	+ O(x^2),\\
	\sigma(\mathcal{K}((f^x)^2))
	& = \sigma_0
	+ O(x^2),\\
	\sigma'(\mathcal{K}((f^x)^2))
	& = \sigma_0'
	+ O(x^2),\\
	\tilde{F}(f^x)
	& = \sigma_0^2
	+ O(x^2),\\
	\tilde{G}(f^x)
	& = \langle \sigma_0, \dot{\alpha} \rangle x
	+ O(x^2).
\end{align*}
Therefore,
\begin{align*}
	\langle \sigma(\mathcal{K}((f^x)^2)), \ind{[0, t]} \rangle
	& = \sigma_0 t + O(x),\\
	2 \langle \sigma'(\mathcal{K}((f^x)^2)), \dot{f}^x \mathcal{K}(f^x \operatorname{id}_{\leq t}) \rangle
	& = O(x),\\
	2 \langle (\sigma^2)'(\mathcal{K}((f^x)^2)), \mathcal{K}(f^x \operatorname{id}_{\leq t}) \rangle
	&
	= O(x),\\
	x - \rho \tilde{G}(f^x)
	& = (1 - \rho \sigma_0 \alpha_1) x
	+ O(x^2),\\
	(x - \rho \tilde{G}(f^x))^2
	& = O(x^2).
\end{align*}
We use the previous formulas in \eqref{eq:minimizing configuration} to obtain
\begin{align}\begin{split}
	f_t^x
	& = \frac{\rho((1 - \rho \sigma_0 \alpha_1) x + O(x^2)) (\sigma_0 t + O(x))}{\bar{\rho}^2(\sigma_0^2 + O(x^2))}
	+ \frac{O(x^2)}{2 \bar{\rho}^2 (\sigma_0^4 + O(x^2))} O(x)\\
	& = \frac{\rho (1 - \rho \sigma_0 \alpha_1) x \sigma_0 t}{\bar{\rho}^2 \sigma_0^2}
	+ O(x^2).
\end{split}\end{align}
Comparing the coefficients, we get the  same result as the authors of \cite{BaFrGuHoSt19} for the first order expansion, i.e.\
\begin{align}
	\alpha_t
	= \frac{\rho (1 - \rho \sigma_0 \alpha_1)}{\bar{\rho}^2 \sigma_0} t.
\end{align}
Setting $ t = 1 $ and then computing $ \alpha_1 $ leads to the formula
\begin{align}
	\alpha_t
	= \frac{\rho}{\sigma_0} t.
\end{align}
Note that the first order expansion of the minimizing configuration $ f^x $ is \emph{exactly} the same as in \cite{BaFrGuHoSt19}. The reason is that the expansions of the ingredients of \eqref{eq:minimizing configuration} are relevant here, and these expansions coincide. For the second order expansion of the rate function, we need second order expansions of its ingredients. These are given in the following formulas, where $\operatorname{id}^2$ denotes
the quadractic function $s\mapsto s^2$:
\begin{align*}
	\frac{1}{2} \tilde{E}(f^x)
	& = \frac{1}{2} \frac{\rho^2}{\sigma_0^2} x^2 + O(x^3),\\
	(x - \rho \tilde{G}(f^x))^2
	& = \bar{\rho}^4 x^2 + O(x^3)\\
	\tilde{F}(f^x)
	& = \sigma_0^2
	+ (\sigma_0^2)' \langle \mathcal{K}(\alpha^2), 1 \rangle x^2
	+ O(x^3)\\
	& = \sigma_0^2 + (\sigma_0^2)' \frac{\rho^2}{\sigma_0^2} \langle \mathcal{K}(\operatorname{id}^2), 1 \rangle x^2 + O(x^3).
\end{align*}

Finally, we get the Taylor expansion of the rate function by taking
into account the reasoning above. We  insert the expansion
\begin{align}
	f_t^x
	= \alpha_t x
	+ O(x^2)
	= \frac{\rho}{\sigma_0} t x 
	+ O(x^2)
\end{align}
and the expansions above into Eq.~\eqref{eq:minimizing configuration} for the minimizing configuration. Then, we get
\begin{align}
	\mathcal{I}_x(f^x)
	\notag & = \frac{(x - \rho \tilde{G}(f^x))^2}{2 \bar{\rho}^2 \tilde{F}(f^x)} + \frac{1}{2} \tilde{E}(f^x)\\
	\notag & = \frac{\bar{\rho}^4 x^2 + O(x^3)}{2 \bar{\rho}^2 \big(\sigma_0^2 + (\sigma_0^2)' \frac{\rho^2}{\sigma_0^2} \langle \mathcal{K}(\operatorname{id}^2), 1 \rangle x^2 + O(x^3)\big)}
	+ \frac{1}{2} \frac{\rho^2}{\sigma_0^2} x^2 + O(x^3)\\
	\notag & = \frac{\bar{\rho}^2}{2 \sigma_0^2} x^2 + O(x^3)
	+ \frac{1}{2} \frac{\rho^2}{\sigma_0^2} x^2 + O(x^3)\\
	\notag & = \frac{1}{2 \sigma_0^2} ( \bar{\rho}^2 + \rho^2) x^2 + O(x^3)\\
	& = \frac{1}{2 \sigma_0^2} x^2 + O(x^3),
\end{align}
and hence the following expansion holds:
\begin{align}
	I(x)
	= \mathcal{I}_x(f^x)
	= \frac{1}{2 \sigma_0^2} x^2 + O(x^3).
\end{align}

\section*{Acknowledgement}
We gratefully acknowledge financial support from the Austrian Science Fund (FWF) under grant P~30750.

%
%
%
%
%
%

%
%
%
%
%
%

%
%
%
%
%
%
\bibliographystyle{siam}
\bibliography{literature}

\begin{thebibliography}{10}

\bibitem{BaFrGuHoSt19}
{\sc C.~Bayer, P.~K. Friz, A.~Gulisashvili, B.~Horvath, and B.~Stemper}, {\em
  Short-time near-the-money skew in rough fractional volatility models}, Quant.
  Finance, 19 (2019), pp.~779--798.

\bibitem{BeFr09}
{\sc S.~Benaim and P.~Friz}, {\em Regular variation and smile asymptotics},
  Math. Finance, 19 (2009), pp.~1--12.

\bibitem{BoDu98}
{\sc M.~Bou\'{e} and P.~Dupuis}, {\em A variational representation for certain
  functionals of {B}rownian motion}, Ann. Probab., 26 (1998), pp.~1641--1659.

\bibitem{BaDe20}
{\sc N.~Bäuerle and S.~Desmettre}, {\em Portfolio optimization in fractional
  and rough {H}eston models}, SIAM Journal on Financial Mathematics, 11 (2020),
  pp.~240--273.

\bibitem{CePa20}
{\sc M.~Cellupica and B.~Pacchiarotti}, {\em Pathwise asymptotics for
  {V}olterra type stochastic volatility models}.
\newblock J {T}heor {P}robab (2020).
  \url{https://doi.org/10.1007/s10959-020-00992-4}.

\bibitem{ChFi14}
{\sc A.~Chiarini and M.~Fischer}, {\em On large deviations for small noise
  {I}t\^{o} processes}, Adv. in Appl. Probab., 46 (2014), pp.~1126--1147.

\bibitem{DeZe98}
{\sc A.~Dembo and O.~Zeitouni}, {\em Large deviations techniques and
  applications}, vol.~38 of Stochastic Modelling and Applied Probability,
  Springer-Verlag, New York, second~ed., 1998.

\bibitem{DuEl97}
{\sc P.~Dupuis and R.~S. Ellis}, {\em A weak convergence approach to the theory
  of large deviations}, Wiley Series in Probability and Statistics: Probability
  and Statistics, John Wiley \& Sons, Inc., New York, 1997.
\newblock A Wiley-Interscience Publication.

\bibitem{FoGeSm19}
{\sc M.~Forde, S.~Gerhold, and B.~Smith}, {\em Small-time and large-time smile
  behaviour for the rough {H}eston model}.
\newblock Preprint, available at \url{https://arxiv.org/abs/1906.09034}, 2019.

\bibitem{FoZh17}
{\sc M.~Forde and H.~Zhang}, {\em Asymptotics for rough stochastic volatility
  models}, SIAM J. Financial Math., 8 (2017), pp.~114--145.

\bibitem{GeJaRoSh18}
{\sc H.~Guennoun, A.~Jacquier, P.~Roome, and F.~Shi}, {\em Asymptotic behavior
  of the fractional {H}eston model}, SIAM J. Financial Math., 9 (2018),
  pp.~1017--1045.

\bibitem{Gu19}
{\sc A.~Gulisashvili}, {\em Gaussian stochastic volatility models: Scaling
  regimes, large deviations, and moment explosions}.
\newblock Stochastic Processes and their Applications (2019),
  \url{https://doi.org/10.1016/j.spa.2019.10.005}.

\bibitem{Gu18}
\leavevmode\vrule height 2pt depth -1.6pt width 23pt, {\em Large deviation
  principle for {V}olterra type fractional stochastic volatility models}, SIAM
  J. Financial Math., 9 (2018), pp.~1102--1136.

\bibitem{Gu20_arxiv}
\leavevmode\vrule height 2pt depth -1.6pt width 23pt, {\em Time-inhomogeneous
  {G}aussian stochastic volatility models: Large deviations and super
  roughness}.
\newblock Preprint, arXiv:2002.05143, 2020.

\bibitem{Hu03}
{\sc H.~Hult}, {\em Approximating some {V}olterra type stochastic integrals
  with application to parameter estimation}, Stochastic Process. Appl., 105
  (2003), pp.~1--32.

\bibitem{Hu03a}
\leavevmode\vrule height 2pt depth -1.6pt width 23pt, {\em Extremal Behavior of
  Regularly Varying Stochastic Processes}, PhD thesis, Royal {I}nstitute of
  {T}echnology, {S}tockholm, 2003.

\bibitem{LiWaYaZh17}
{\sc Y.~Li, R.~Wang, N.~Yao, and S.~Zhang}, {\em A moderate deviation principle
  for stochastic {V}olterra equation}, Statist. Probab. Lett., 122 (2017),
  pp.~79--85.

\bibitem{NuRo00}
{\sc D.~Nualart and C.~Rovira}, {\em Large deviations for stochastic {V}olterra
  equations}, Bernoulli, 6 (2000), pp.~339--355.

\bibitem{SaKiMa93}
{\sc S.~G. Samko, A.~A. Kilbas, and O.~I. Marichev}, {\em Fractional Integrals
  and Derivatives}, {G}ordon and {B}reach {S}cience {P}ublishers, 1993.

\bibitem{Zh08}
{\sc X.~Zhang}, {\em Euler schemes and large deviations for stochastic
  {V}olterra equations with singular kernels}, J. Differential Equations, 244
  (2008), pp.~2226--2250.

\bibitem{Zh10}
\leavevmode\vrule height 2pt depth -1.6pt width 23pt, {\em Stochastic
  {V}olterra equations in {B}anach spaces and stochastic partial differential
  equation}, J. Funct. Anal., 258 (2010), pp.~1361--1425.

\end{thebibliography}

\end{document}